\documentclass[10pt]{article}
\usepackage{fullpage}
\usepackage{epsfig}
\usepackage{amsmath}
\usepackage{amsthm}
\usepackage{amssymb}
\usepackage{latexsym}
 \numberwithin{equation}{section}
\theoremstyle{definition}
\newtheorem{example}{Example}[section]
\newtheorem{definition}[example]{Definition}
\newtheorem{remark}[example]{Remark}
\theoremstyle{plain}
\newtheorem{theorem}[example]{Theorem}

\newtheorem{lemma}[example]{Lemma}
\newtheorem{corollary}[example]{Corollary}

\newcommand{\R}{{\mathbb R}}
\newcommand{\N}{{\mathbb N}}
\newcommand{\ups}{\Upsilon_{{\cal R}}^p(\R^{m\times n})}
\newcommand{\upss}{\Upsilon_{{\cal S}}^p(\R^{m\times n})}

\newcommand{\be}{\begin{eqnarray}}

\newcommand{\coms}{{\beta_{{\cal S}}\R^{m\times n}}}
\newcommand{\ee}{\end{eqnarray}}

\renewcommand{\d}{{\rm d}}
\newcommand{\cdm}{{\cal DM}^p_{\cal R}(\O;\R^{m\times n})}
\newcommand{\cdms}{{\cal DM}^p_{\cal S}(\O;\R^{m\times n})}
\newcommand{\gcdm}{{\cal GDM}^p_{\cal R}(\O;\R^{m\times n})}
\newcommand{\dm}{{\cal DM}^p_{\cal R}(\O;\R^{m\times n})}

\newcommand{\md}{{\rm d}}

\renewcommand{\O}{\Omega}
\newcommand{\s}{\sigma}

\renewcommand{\b}{\beta}

\newcommand{\Rn}{\R^{n}}

\newcommand{\wto}{\rightharpoonup}

\newcommand{\rca}{{\rm rca}}
\newcommand{\prca}{{\rm rca}_{\vspace*{.5mm}1}^+}

\newcommand{\cN}{{\mathcal N}}
\renewcommand{\div}{\operatorname{div}}

\newcommand{\abs}[1]{\left\vert#1\right\vert}

\newcommand{\eps}{\varepsilon}

\begin{sloppypar}

\title{Sequential weak continuity of  null Lagrangians at the boundary \thanks{This work  was   supported by the grants P201/12/0671 and   P201/10/0357 (GA \v{C}R).
}}

\author{ {Agnieszka} Ka\l amajska\thanks{Institute of Mathematics, Warsaw University, ul.~Banacha 2, PL-02-097 Warsaw, Poland. The work is supported by the  Polish Ministry of
Science grant no. N N201 397837 (years 2009-2012). }\hspace{3mm}
Stefan Kr\"omer\thanks{Mathematisches Institut, Universit\"at zu K\"oln, 50923 K\"oln, Germany} \hspace{3mm} Martin
Kru\v{z}\'{\i}k\thanks{Institute of Information Theory and Automation of the ASCR, Pod vod\'{a}renskou
v\v{e}\v{z}\'{\i}~4, CZ-182~08~Praha~8, Czech Republic (corresponding
address) \& Faculty of Civil Engineering, Czech Technical
University, Th\'{a}kurova 7, CZ-166~ 29~Praha~6, Czech Republic  ({\tt
kruzik@utia.cas.cz})} }

\begin{document}
\date{}
\maketitle

\bigskip
%\noindent

\bigskip

\begin{abstract}
We show weak* in measures on $\bar\O$/ weak-$L^1$ sequential continuity of $u\mapsto f(x,\nabla u):W^{1,p}(\O;\R^m)\to L^1(\O)$, where
$f(x,\cdot)$ is a null Lagrangian for $x\in\O$,  it is a null Lagrangian at the boundary for $x\in\partial\O$ and $|f(x,A)|\le C(1+|A|^p)$. We also give a precise characterization
of null Lagrangians at the boundary in arbitrary dimensions. Our results explain, for instance,  why $u\mapsto \det\nabla u:W^{1,n}(\O;\R^n)\to L^1(\O)$ fails to be weakly continuous. Further, we state a new weak lower semicontinuity theorem for integrands depending on null Lagrangians at the boundary. The paper closes with an example indicating that a well-known result on higher integrability of determinant \cite{Mue89a} need not  necessarily extend to our setting.  The notion of  quasiconvexity at the boundary due to J.M.~Ball and J.~Marsden is central to our analysis.
\end{abstract}

\medskip

%\noindent
\medskip

%\noindent

%\tableofcontents

\section{Introduction}
This paper is inspired  by the well-known example  \cite[Example 7.3]{BaMu91} or   \cite[Example 8.6]{Da89B} showing that if $\O\subset\R^2$ is bounded and Lipschitz and $\{u_k\}\subset W^{1,2}(\O;\R^2)$ weakly converges to the origin then, in general, $\int_\O\det\nabla u_k(x)\,\md x\not\to 0$  which means that  $\det\nabla u_k\not\wto 0$ in $L^1(\O)$, neither $\det\nabla u_k\stackrel{*}{\wto} 0$ in $\rca(\bar\O)$ (Radon measures on $\bar\O$).  Contrary to that, if the sequence were bounded in $W^{1,p}(\O;\R^2)$ for $p>2$ then $\{\det\nabla u_k\}_{k\in\N}$ would weakly tend to zero in $L^1(\O)$.  Therefore, a natural question arises which  functions $f:\R^{m\times n}\to\R$, $|f(A)|\le C(1+|A|^p)$, have the property that $u\mapsto f(\nabla u)$ is (weakly,weakly*) sequentially continuous as maps from $W^{1,p}(\O;\R^m)$ to $\rca(\bar\O)$, $p>1$. It is obvious that such functions must be  quasiaffine, i.e., $f$ is an affine function of all subdeterminants of its argument \cite{Da89B}, however, as the above mentioned example shows, it is far from being sufficient.  It turns out that this question  is intimately related to concentrations of $\{|\nabla u_k|^p\}_{k\in\N}\subset L^1(\O)$ at the boundary of $\O$ and that, for a general domain $\O$, $f$ must also depend on $x\in\O$.  We also show that the notion of quasiconvexity at the boundary, introduced in \cite{BaMa84a} to study necessary  conditions for local minimizers of variational integral functionals  plays a key role in our analysis.

The plan of the paper is as follows. After introducing necessary notation we recall the notions of quasiconvexity and quasiconvexity at the boundary. Then we explicitly characterize all functions which, together with their negative multiple, are quasiconvex at the boundary. These are here called null Lagrangians at the boundary. Our characterization is a slight  adaptation of the result of P.~Sprenger \cite{Spre96B} which does not seems  to be  well-known to the calculus-of-variations  community.  We state our main result Theorem~\ref{th:weakcontinuityup} using a recently discovered characterization of DiPerna-Majda measures generated by gradients and get a new weak lower semicontinuity result for integral functionals depending on  null Lagrangians at the boundary. Finally, we construct an example indicating  that a result analogous to higher integrability of determinants due to M\"{u}ller \cite{Mue89a} may not hold for null Lagrangians at the boundary.

\bigskip

\section{Basic notation.}
Let us  start with a few definitions and with the explanation of our notation.
Having a bounded domain $\O\subset\R^n$ we denote by $C(\O)$ the space of continuous functions from $\O$ to $\R$.  Then $C_0(\O)$ consists of  functions from $C(\O)$ whose support is contained in $\O$.
 More generally, for any topological space $S$, by $C(S)$ we denote all continuous functions on $S$.
In what follows ``{\rm rca}$(S)$'' denotes the set of regular countably additive set functions on the Borel $\s$-algebra on a metrizable set  $S$ (cf. \cite{d-s}), its subset, {\rm rca}$^+_1(S)$,  denotes regular  probability measures on a set $S$.
We write ``$\gamma$-almost all'' or ``$\gamma$-a.e.'' if  we mean ``up to a set with the $\gamma$-measure zero''. If $\gamma$ is the $n$-dimensional Lebesgue measure and $M\subset\R^n$ we omit writing $\gamma$ in the notation.
Further, $W^{1,p}(\O;\R^m)$, $1\le p<+\infty$  denotes the usual space of measurable mappings which are together with
their first (distributional) derivatives integrable with the $p$-th power.
The support of a measure $\sigma\in\ {\rm rca}(\O)$ is a smallest closed set $S$
such that $\sigma(A)=0$ if $S\cap A=\emptyset$. Finally, if $\sigma\in\rca(S)$ we write
$\sigma_s$ and $d_\sigma$ for the singular part and density  of $\sigma$ defined by   the Lebesgue decomposition, respectively. We denote by `w-$\lim$' the weak limit and by $B(x_0,r)$ an open ball in $\R^n$ centered at $x_0$ and the radius $r>0$. The scalar product on $\R^n$ is standardly defined as $a\cdot b:=\sum_{i=1}^na_ib_i$ and analogously on $\R^{m\times n}$. Finally, if $a\in\R^m$ and $b\in\R^n$ then  $a\otimes b\in\R^{m\times n}$ with $(a\otimes b)_{ij}=a_ib_j$, and $\mathbb{I}$ denotes the identity matrix.
%If not explicitely stated otherwise, we will suppose in the sequel  that
%$\O\subset\R^n$ is a bounded domain with a $C^1$ boundary.
\bigskip

\subsection{Quasiconvex functions}
\noindent
Let $\O\subset\R^n$ be a bounded domain. We say that  a function $v:\R^{m\times n}\to\R$ is quasiconvex \cite{Mo52a} if
for any $F\in\R^{m\times n}$ and any $\varphi\in W^{1,\infty}_0(\O;\R^m)$
\be\label{quasiconvexity}
v(F)|\O|\le \int_\O v(F+\nabla \varphi(x))\,\md x\ .\ee
If $v:\R^{m\times n}\to\R$ is not quasiconvex we define its quasiconvex envelope
$Qv:\R^{m\times n}\to\R$ as
$$
Qv=\sup\left\{h\le v;\ \mbox{$h:\R^{m\times n}\to\R$ quasiconvex }\right\}\ $$
and if the set on the right-hand side is empty we put $Qv=-\infty$.
If $v$ is locally bounded and Borel measurable then for any $F\in\R^{m\times n}$ (see \cite{Da89B})
\be\label{relaxation}
Qv(F)=\inf_{\varphi\in W^{1,\infty}_0(\O;\R^m)} \frac{1}{|\O|} \int_\O v(F+\nabla \varphi(x))\,\md x\ .\ee

We will also need the following  elementary  result. It can be found in a more general form   e.g. in \cite[Ch.~4, Lemma~2.2]{Da89B}  or in \cite{Mo52a}.

\begin{lemma}\label{lemma}  Let $v:\R^{m\times n}\to\R$ be quasiconvex  with  $|v(F)|\le C(1+|F|^p)$, $C>0$, for all $F\in\R^{m\times n}$.
Then there is a constant $\alpha\ge 0$ such that  for every  $F_1,F_2\in\R^{m\times n}$ it holds
\be\label{p-lipschitz-gen}
|v(F_1)-v(F_2)|\le \alpha(1+|F_1|^{p-1}+ |F_2|^{p-1})|F_1-F_2|\ .\ee
\end{lemma}

\bigskip

Following \cite{BaMa84a,Sil97B, Spre96B} we define the notion of  quasiconvexity at the boundary. In order to proceed, we first define the so-called {\it standard boundary domain}.

\begin{definition}\label{def:stabounddom}
Let $\varrho\in\R^n$ be a unit vector and let $\O_\varrho\subset \R^n$ be a bounded  Lipschitz  domain. We say that
$\O_\varrho$ is a standard boundary domain with the normal $\varrho$ if there is $a\in\R$ such that $\O_\varrho\subset H_{a,\varrho}:=\{x\in\R^n;\ \varrho\cdot x<a\}$ and the $(n-1)$- dimensional interior ${\Gamma_\varrho}$ of ${\partial\O_\varrho\cap\partial H_{a,\varrho}}$ is not empty.
\end{definition}

%\red{
%\begin{rem}\rm Note that after the rotation of coordinates such that $A_\rho =(1,0,\dots )$, $H_{a,\rho}=\{ x\in \R^n : x_1<0\}$. In particular $\partial\Omega$ must contain ``flat part'' of the boundary with the normal $\rho$.
%\end{rem}
%}

\noindent

 For $1\le p\le+\infty$, and any bounded Lipschitz domain $\O$, we define
\be\label{testspace}
 W^{1,p}_{\partial\O\setminus\Gamma}(\O;\R^m):=\{ u\in W^{1,p}(\O;\R^m);\ u\equiv 0 \mbox{ on } \partial\O\setminus\Gamma\} ,\ee
where the condition $u\equiv 0$ is understood in the sense of operator of trace, in particular the equality holds ${\cal H}^{n-1}$-almost everywhere with respect to the
$n-1$-dimensional Hausdorff measure on $\partial\Omega$.

We are now ready to define the quasiconvexity at the boundary.

\begin{definition} (\cite{BaMa84a})\label{qcb-def}
Let $\varrho\in\R^n$ be a unit vector, and let $v:\R^{m\times n}\to\R$ be a given function.
\begin{description}
\item[i)]
$v$ is called quasiconvex at the boundary at $F\in\R^{m\times n}$ (where  $F\in\R^{m\times n}$ is given), with respect to $\varrho$ (shortly $v$ is qcb at $(F,\varrho)$), if there is $q\in\R^m$ such that {for every standard boundary domain $\O_\varrho$ with the normal $\varrho$}
and for every $u\in {W^{1,\infty}_{\partial\O_\varrho\setminus\Gamma_\varrho}(\O_\varrho;\R^m)}$, we have
\be\label{qcbinfty}
{\int_{\Gamma_\varrho} q\cdot u(x)\,\md S + v(F)|\O_\varrho|\le \int_{\O_\varrho} v(F+\nabla u(x))\,\md x\ .} \ee
\item[ii)]  $v$ is called quasiconvex at the boundary if it is  quasiconvex at the boundary at every $F\in\R^{m\times n}$ and every $\varrho\in\R^n$.
\end{description}
\end{definition}
\bigskip

\noindent
An immediate generalization of the above definition is the following one.

\bigskip

\begin{definition}\label{p-qcb-def}
Let $\varrho\in\R^n$ be a unit vector, $F\in\R^{m\times n}$, $1\le p<+\infty$, $v:\R^{m\times n}\to\R$ is such that $|v|\le C(1+|\cdot|^p)$ for some $C>0$.
\begin{description}
\item[i)]
A function $v$ is called $W^{1,p}$-quasiconvex at the boundary at given $F\in\R^{m\times n}$ with respect to $\varrho$ (shortly $v$ is $p$-qcb at $(F,\varrho)$), if there is $q\in\R^m$ such that {for every standard boundary domain $\O_\varrho$ with the normal $\varrho$}
and for every $u\in W^{1,p}_{\partial\O_\varrho\setminus \Gamma_\varrho}(\O_\varrho;\R^m )$ , we have
\be\label{qcb}
\int_{\Gamma_\varrho} q\cdot u(x)\,\md S + v(F)|\O_\varrho|\le \int_{\O_\varrho} v(F+\nabla u(x))\,\md x\ . \ee
\item[ii)] A function $v$ is called $W^{1,p}$-quasiconvex at the boundary if it is  $W^{1,p}$-quasiconvex at the boundary at every $F\in\R^{m\times n}$ and every $\varrho\in\R^n$.
\end{description}
\end{definition}

\noindent
Let us formulate several remarks, concerning the notation of functions quasiconvex at the boundary.

\begin{remark}\label{remark1}
\begin{description}
\item[(i)] If $v$ is differentiable {at $F$} then vector $q$ satisfying (\ref{qcbinfty}) is uniquely defined and
$q={\nabla v}(F)\varrho$,
 cf.~\cite{Spre96B}.
%Let us denote the set of all  vectors $q$ for which are admissible for (\ref{qcbinfty}) by
%$\partial_v^{\rm qcb}(F,\varrho)$. This set may be considered as a notion of a ``subdifferential'' for $v$. Then condition (ii) %from \cite[Th.~2.2]{BaMa84a}
%reads
%$0\in  \partial_v^{\rm qcb}(\nabla u(x_0),\varrho)$ where $x_0\in\partial\O$ and $\varrho$ is the outer unit normal to %$\partial\O$ at $x_0$. \Here $u\in W^{1,p}(\O;\R^m)$ is a local minimizer of class $C^1$ of the functional %$u\mapsto\int_\O v(\nabla u(x))\,\md x$.
\item[(ii)] It is clear that if $v$ is qcb at $(F,\varrho)$ it is also quasiconvex at $F$, i.e., (\ref{quasiconvexity}) holds.
\item[(iii)] If (\ref{qcbinfty}) holds for one standard boundary domain it holds for  other standard boundary domains {with the normal $\rho$}, too, \cite{BaMa84a}.
\item[(iv)] {If $p>1$,  $v:\R^{m\times n}\to\R$ is positively $p$-homogeneous, i.e.~$v(\lambda F)=\lambda^pv(F)$ for all $F\in\R^{m\times n}$, continuous, and   $p$-qcb at $(0,\varrho)$ then $q=0$ in (\ref{qcbinfty}).
 Indeed,
%\footnote{more direct proof}
as $v(t\theta)=t^pv(\theta)$, then $\partial_{\theta}v=\lim_{t\to 0}\frac{v(t\theta)-v(0)}{t}=\lim_{t\to 0} t^{p-1}v(\theta)=0$, whenever $\theta\in \R^{m\times n}$. Therefore
    $v$ is differentiable at $0$ with $\nabla v(0)=0$ and according to our Remark (i), $q=\nabla v(0)\rho=0$.
    Moreover, let us note that (\ref{qcbinfty})  implies that $\int_{\O_\varrho} v(\nabla u(x))\,\md x\ge0$.}
\item[(v)] Under the growth assumption $|v|\le C(1+|\cdot|^p)$ for some $1\le p<+\infty$ and $C>0$, $W^{1,p}$-quasiconvexity at the boundary is equivalent to the
the quasiconvexity at the boundary \cite{Kru10a}.
\item[(vi)] We refer an interested reader to \cite{GrMe07a,GrMe08a,MiSp98} for other applications of quasiconvexity at the boundary in variational context.
\end{description}
\end{remark}

\bigskip
\noindent
It will be convenient to define the following notion of quasiconvex at the boundary envelope of $v$ at zero. Note that we integrate only over a standard boundary domain with a given normal.

\begin{definition}\label{quasienvel}
{Let $\Omega_\varrho\subset \R^n$ be the standard boundary domain with the normal $\varrho\in\R^n$ of the unit length and let $\Gamma_\varrho$ be as in Definition \ref{def:stabounddom}. Let $v:\R^{m\times n}\to\R$ be continuous and positively $p$-homogeneous.
By the $W^{1,p}$-quasiconvex envelop at the boundary at $0$, we define the quantity:
\be\label{envelope}
Q_{b,\varrho}v(0):=\inf_{u\in W_{\O_\varrho\setminus\Gamma_\varrho}^{1,p}(\O_\varrho;\R^m)}\frac{1}{|\O_\varrho|}\int_{\O_\varrho}v(\nabla u(x))\,\md x\ .\ee
}
\end{definition}

\bigskip
\noindent
Below we state an example of function which is quasiconvex at the boundary.

\begin{example}
It is shown in \cite[Prop.~17.2.4]{Sil97B} that the function ${v:\R^{3\times 3}\to\R}$ given  by
$$
v(F):=a\cdot [{\rm Cof} F]\varrho\ $$
is quasiconvex at the boundary with the unit normal $\varrho\in\R^3$. Here $a\in\R^3$ is an arbitrary  constant and ``Cof'' is the cofactor matrix, i.e., $[{\rm Cof }F]_{ij}=(-1)^{i+j}{\rm det }F'_{ij}$, where ${F'_{ij}\in\R^{2\times 2}}$ is the submatrix of $F$ obtained from $F$ by removing the $i$-th row and the $j$-th column. Hence,  $v$ is positively $2$-homogeneous.  This particular function is also called an interface null Lagrangian in \cite{Sil2011}.
\end{example}

\subsection{Null Lagrangians at the boundary}

\begin{definition}\label{nulllagrangian}
Let $\varrho\in\R^n$ be a unit vector and let $v:\R^{m\times n}\to\R$ be a given function.
\begin{description}
\item[i)]
$v$ is called a {\it null Lagrangian at the boundary} at given $F\in \R^{m\times n}$
; cf.~\cite{Sil97B},
if  both  $v$ and $-v$ are quasiconvex at the boundary at  $F$,
i.e., there exists $q\in\R^m$ such that
{for every standard boundary domain $\O_\varrho$ with the normal $\varrho$} and for all $u\in W^{1,p}_{\partial\O_\varrho\setminus \Gamma_\varrho}(\O_\varrho;\R^m )$, we have
 \be\label{null-identity}
\int_{\Gamma_\varrho} q\cdot u(x)\,\md S + v(F)|\O_\varrho| = \int_{\O_\varrho} v(F+\nabla u(x))\,\md x\ . \ee
 \item[ii)] If $v$ is a {\it null Lagrangian at the boundary} at every $F\in \R^{m\times n}$, we call it a null Lagrangian at the boundary.
\end{description}
\end{definition}

\begin{definition}
Let  $\varrho\in\R^n$ be a unit vector. A mapping $\cN:\R^{m\times n}\to\R$ will be called a special null Lagrangian at the boundary at
$F\in\R^{m\times n}$, with respect to the normal $\varrho$ if for every  $W^{1,\infty}_{\partial\O_\varrho\setminus\Gamma_\varrho}(\O_\varrho;\R^m)$ it holds that
\be\label{null-identity1}
 \int_{\O_\varrho}\mathcal{N}(F+\nabla u(x))\,\md x=\mathcal{N}(F)|{\O_\varrho}|\ .
\ee
In particular, equation (\ref{null-identity}) holds with $q=0$.
\end{definition}
%\footnote{I added name ``special'', I think that the transformation does not help %as we deal with any $F$}

\begin{remark}
Given a fixed $F\in \R^{m\times n}$,
every null Lagrangian at the boundary at $F_0$ can be transformed into
a special null Lagrangian at the boundary at $F_0$ by adding a linear term.
More precisely, we have the following result:
If $\cN$ is a null Lagrangian at the boundary at $F$ with normal $\varrho$,
then $\cN$ is differentiable (in fact, it is a null Lagrangian and thus a polynomial), and according to Remark~\ref{remark1} (i), the vector $q$ in
\eqref{null-identity} is given by
$
	q=\frac{\partial}{\partial F}\cN(F_0)\varrho.
$
If we define $\tilde{\cN}:\R^{m\times n}\to \R$ by
$$
	\tilde{\cN}(F):=\cN(F)-\frac{\partial}{\partial F}\cN(F_0)\cdot F,~~F\in \R^{m\times n},
$$
we have that $\frac{\partial}{\partial F}\tilde{\cN}(F_0)=0$, and consequently, $\tilde{\cN}$ is
a special null Lagrangian at the boundary at $F_0$.

\end{remark}
In view of the previous remark,
the following theorem explicitly characterizes all possible null Lagrangians at the boundary. It was first proved by P.~Sprenger in his thesis \cite[Satz 1.27]{Spre96B}
written in German. We give here his original proof with some minor simplifications. Before stating the result we recall that
${\rm SO}(n):=\{R\in\R^{n\times n};\ R^\top R=RR^\top=\mathbb{I}\ ,\ \det R=1\}$ denotes the set of orientation-preserving rotations and if we write $A=(B|\varrho)$ for some $B\in\R^{n\times(n-1)}$ and $\varrho\in\R^n$ then $A\in\R^{n\times n}$, its last column is $\varrho$ and $A_{ij}=B_{ij}$ for $1\le i\le n$ and $1\le j\le n-1$.

\bigskip

\begin{theorem}\label{thm:bnulllag}
Let $\varrho\in\R^n$ be a unit vector
and let $\mathcal{N}:\R^{m\times n}\to\R$ be a given continuous function.
Then the following four statements are equivalent.
\begin{itemize}
\item[(i)] $\cN$ satisfies \eqref{null-identity1} for every $F\in \R^{m\times n}$;
\item[(ii)] $\cN$ satisfies \eqref{null-identity1} for $F=0$, i.e., $\cN$ is a special null Lagrangian at $0$;
\item[(iii)] There are constants $\tilde\beta_s\in\R^{{m\choose s}\times {{n-1}\choose s}}$, $1\le s\le\min(m,n-1)$, such that for all $H\in\R^{m\times n}$,
\be\label{null-form}
	\cN(H)=\cN(0)+\sum_{s=1}^{\min(m,n-1)}\tilde\beta_s\cdot {\rm ad}_s (H \tilde R),
\ee
where  $\tilde R\in\R^{n\times(n-1)}$ is a matrix such that $R=(\tilde R|\varrho)$ belongs to ${\rm SO}(n)$;
\item[(iv)] $\cN(F+a\otimes \varrho)=\cN(F)$ for every $F\in \R^{m\times n}$
and every $a\in \R^m$.
\end{itemize}
\end{theorem}

\begin{remark}
Condition (ii) is not a part of the statement of \cite[Satz 1.27]{Spre96B}, but it follows from the proof. On the other hand, we omitted a simple variant of (iv) in terms of the derivative of $\cN$ that was given by Sprenger.
\end{remark}

\begin{proof}[Proof of Theorem~\ref{thm:bnulllag}]
At first we note that the proof can be reduced to the case when $\varrho =e_n$ in the formulation of the statements.
Indeed, let
 $e_1,\ldots,e_n$ denote the standard unit vectors in $\R^n$ and observe that $R e_n=\varrho$, with $R$ defined in (iii).
Moreover, let $\O_{e_n}:=R^\top \O_\varrho$, $\Gamma_{e_n}:= R^\top \Gamma_\varrho$ and $\hat{\cN}(F):=\cN(F R^\top)$ for $F\in \R^{m\times n}$, and let
$\tilde{F}:= FR$. Note that $\hat{\cN}$ also is a null Lagrangian, and $\O_{e_n}$ is a standard boundary domain to the normal vector $e_n$.
By a change of variables, \eqref{null-identity} is equivalent to
\be\label{tbnl-1}
 \int_{\O_{e_n} }
 \hat{\cN}(\tilde{F}+\nabla u(x))\,\d x=\hat\cN (\tilde{F}) |\O_{e_n}|\
 ~~\text{for every $u\in W^{1,\infty}_{\partial\O_{e_n}\setminus\Gamma_{e_n}}(\O_{e_n};\R^m)$}.
\ee
An easy verification shows that (i)-(iv) defined for $\cN$ are equivalent to the analogous counterparts for $\hat{\cN}$.

As a consequence, it suffices to prove the assertion for the case $\cN=\hat{\cN}$, $\varrho=e_n$ and $\tilde{R}=(e_1|\ldots|e_{n-1})\in \R^{n\times (n-1)}$.
Moreover, as (i)--(iv) clearly remain unchanged if we add a constant to $\cN$, therefore we may assume that $\cN(0)=0$.

Since (i) obviously implies (ii), we only have to show that (ii) $\Rightarrow$ (iii) $\Rightarrow$ (iv) $\Rightarrow$ (i).

\noindent{\bf (ii) implies (iii):}\\
Since $\mathcal{N}$ is a null Lagrangian with $\cN(0)=0$, there are constants $\beta_s\in\R^{{m\choose s}\times {{n}\choose s}}$, $1\le s\le\min(m,n)$ such that
\be\label{null}
\mathcal N(H)=\sum_{s=1}^{\min(m,n)}\beta_s\cdot {\rm ad}_s H .
\ee
(see e.g. \cite{Da89B}).
By (ii), for every $\varphi\in W^{1,\infty}_{\partial\O_{e_n}\setminus \Gamma_{e_n}}(\O_{e_n};\R^m)$ we have that
$$
	0=\int_{\O_{e_n}}\mathcal{N}(\nabla\varphi(y))\,\d y,
$$
and using \eqref{null}, we infer that
$$
0=\sum_{s=1}^{\min(m,n)}\beta_s\cdot\int_{\O_{e_n}} {\rm ad}_s \nabla \varphi(y)\,\d y.
$$
As this must hold for all admissible mappings $\varphi$ and ${\rm ad}_s$ is positively homogeneous of degree $s$, by rescaling $\varphi$ it is easy to see that in fact,
\be \label{tbnl-2}
	0=\beta_s\cdot \int_{\O_{e_n}} {\rm ad}_s \nabla \varphi(y)\,\md y\quad\text{for each $s=1,\ldots,\min(m,n)$}.
\ee
Let $I^r_s$ denote the set of all $(\alpha):=\{\alpha_1,\ldots,\alpha_s\}\subset \N$ with $1\le\alpha_1<\alpha_2<\ldots<\alpha_s\le r$.
For $(p)\in I^m_s$ and $(q)\in I^n_s$ we write $\nabla_{(q)} \varphi_{(p)}(x)=\Big(\frac{\partial}{\partial x_{q_j}} \varphi_{p_k}(x)\Big)_{jk} \in \R^{s\times s}$.
With this notation and the divergence structure of determinants, the entries of ${\rm ad}_s$ are defined\footnote{The standard definition requires an additional factor $(-1)^{p+q}$, where $p$ and $q$ denote positions of $(p)$ and $(q)$ in an appropriate ordering of the elements of $I^m_s$ and $I^n_s$, respectively. However, as this factor plays no role in the proof (and could be absorbed into the corresponding constant, anyway), we omit it here.} as
$$
	{\rm ad}_s^{(p)(q)}\nabla \varphi=\det \nabla_{(q)} \varphi_{(p)}.
$$
If $s=1$, $(p)=\{p_1\}$ and $(q)=\{q_1\}$ for some integers $p_1,q_1$, and integration by parts in \eqref{tbnl-2} gives
\be \label{tbnl-2b}
	0=\sum_{p_1=1}^{m}\sum_{q_1=1}^{n}
	\beta_1^{(p)(q)}
	\int_{\Gamma_{e_n}} \varrho_{q_1} \varphi_{p_1}\,\d S,
\ee
where $\varrho$ is the outer normal to $\Gamma_{e_n}$ and $\varrho_{q_1}=\varrho\cdot e_{q_1}$.
In our case $\varrho\equiv e_n$ on $\Gamma_{e_n}$, the flat part of the boundary of $\O_{e_n}$, where $\varphi$ is not subject to a Dirichlet boundary condition.
Hence, all terms below the integral in \eqref{tbnl-2b} vanish unless $q_1=n$, and since $\varphi$ is arbitrary,
we get that
\be\label{tbnl-3}
	\beta^{(p)(q)}_1=0\qquad\text{for every $(p)\in I^m_1$ and $(q)=\{n\}$}.
\ee
In case $s\geq 2$, we can use the divergence structure of determinants as follows:
$$
	{\rm ad}_s^{(p)(q)}\nabla \varphi=\det \nabla_{(q)} \varphi_{(p)}=
	\sum_{i=1}^s \frac{d}{dx_i} \Big[(-1)^{i+s} \varphi_{p_s}\det \big(\nabla_{(q)\setminus \{q_i\}} \varphi_{(p)\setminus \{p_s\}}\big)\Big].
$$
Integrating by parts in \eqref{tbnl-2} yields that
\be \label{tbnl-4}
	\beta_s\cdot \int_{\O_{e_n}} {\rm ad}_s \nabla \varphi(y)\,\md y
	=\sum_{(p)\in I^m_s}\sum_{(q)\in I^n_s} \beta^{(p)(q)}_s\int_{\Gamma_{e_n}}\sum_{i=1}^s(-1)^{i+s}  \varrho_{q_i}
	\varphi_{p_s}
	\det \big(\nabla_{(q)\setminus \{q_i\}} \varphi_{(p)\setminus \{p_s\}}\big)\,\d S.
\ee
As  $\varrho=e_n$ on $\Gamma_{e_n}$,
the inner sum in \eqref{tbnl-4} only contributes if $i=s$ and $q_s=n$,
because otherwise, $q_i<n$ and thus $\varrho_{q_i}=0$.
Denoting
$$
	\bar I^n_{s}:=\{ (\alpha)=\{\alpha_1,\ldots,\alpha_{s}\}\subset \N\mid 1\le\alpha_1<\ldots<\alpha_{s}=n\},
$$
we can combine \eqref{tbnl-2} and \eqref{tbnl-4} to get
$$
	0
	=
	\sum_{(p)\in I^m_s}\sum_{(q)\in \bar I^n_{s}}
	\beta^{(p)(q)}_s\int_{\Gamma_{e_n}}\varphi_{p_s} \det \big(\nabla_{(q)\setminus\{n\}} \varphi_{(p)\setminus \{p_s\}}\big)\,\d S .
$$
As we are free to choose the vector-valued function $\varphi$ with arbitrary components vanishing, this implies that for every $(p)\in I^m_s$,
$$
	0
	=
	\sum_{(q)\in \bar I^n_{s}}
	\beta^{(p)(q)}_s\int_{\Gamma_{e_n}}\psi \det \big(\nabla_{(q)\setminus \{n\}} \eta\big)\,\d S,
$$
for every $\psi\in W^{1,\infty}_{\partial{\O_{e_n}}\setminus\Gamma_{e_n}}(\O_{e_n})$ and every $\eta\in W^{1,\infty}_{\partial{\O_{e_n}}\setminus\Gamma_{e_n}}(\O_{e_n};\R^{s-1})$, and since $\psi$ is arbitrary on $\Gamma_{e_n}$,
we get that	
$$
	0
	=
	\sum_{(q)\in \bar I^n_{s}}
	\beta^{(p)(q)}_s \det \nabla_{(q)\setminus\{n\}} \eta(y)\quad\text{for a.e.~$y\in \Gamma_{e_n}$ (with respect to the surface measure).}
$$
For any given $(q)\in \bar I^n_{s}$, it
is not difficult to find an admissible function $\eta(y)$ which, on some neighborhood of a point in $\Gamma_{e_n}$, only depends on $y_{q_1},\ldots,y_{q_{s-1}}$,
such that $\det \nabla_{(q)\setminus\{n\}} \eta\not\equiv 0$ on $\Gamma_{e_n}$. Together with \eqref{tbnl-3},
we conclude that for $s=1,\ldots,\min(m,n-1)$,
$$
	\beta^{(p)(q)}_s=0\qquad\text{for every $(p)\in I^m_s$ and every $(q)\in \bar I^n_{s}$}.
$$
\newline
Plugging this into \eqref{null}, we obtain \eqref{null-form} for $\tilde{R}=(e_1|\ldots|e_{n-1})$,
with $\tilde{\beta}_s^{(p)(q)}:=\beta_s^{(p)(q)}$ for every $(p)\in I^m_s$ and every $(q)\in I^{n-1}_s=I^{n}_s\setminus \bar I^n_{s}$.

\noindent{\bf (iii) implies (iv):}\\
This is a simple consequence of the fact that $(F+a\otimes\varrho)\tilde{R}=F\tilde{R} +a\otimes(\tilde{R}^\top \varrho)=F\tilde{R}$,
where we used that $(\tilde R|\varrho)\in O(n)$ and thus $\tilde{R}^\top \varrho=0$.

\noindent{\bf (iv) implies (i):}\\
For given $F\in \R^{m\times n}$ and $\varphi\in W^{1,\infty}_{\Gamma_\varrho}(\O_\varrho;\R^m)$,
we define
$$
	g(t):=\int_{\O_\varrho} \cN(F+t\nabla \varphi)\,\d x,\quad t\in\R.
$$
Integrating by parts, we obtain that
\be\label{tbnl-5}
\begin{aligned}
	g'(t)
	&=\int_{\O_\varrho} \nabla_F\cN(F+t\nabla \varphi)\cdot \nabla \varphi\,\d x\\
	&=-\int_{\O_\varrho} [\div \nabla_F\cN(F+t\nabla \varphi)] \cdot \varphi\,\d x
	+\int_{\partial\O_\varrho} \nabla_F\cN(F+t\nabla \varphi) \cdot(\varphi\otimes \varrho)\,\d S
\end{aligned}
\ee
Since $\cN$ is null Lagrange, we know that $\div \nabla_F\cN(F+\nabla \psi)=0$ a.e.~in $\O_\varrho$ for every $\psi\in C^2(\O_\varrho)$.
Hence, the first term in \eqref{tbnl-5} vanishes, and since $\varphi=0$ on $\partial\O_\varrho\setminus \Gamma_\varrho$, we see that
$$
\begin{aligned}
	g'(t)
	&=\int_{\Gamma_\varrho} \nabla_F\cN(F+t\nabla \varphi)\cdot(\varphi\otimes \varrho)\,\d S
\end{aligned}
$$
On the other hand, (iv) implies that
$\nabla \cN(H)\cdot (a\otimes \varrho)=0$ for every $a\in \R^N$ and every $H\in \R^{m\times n}$.
As a consequence, $g'(t)=0$ for every $t\in\R$, whence $g(0)=g(1)$.
\end{proof}

\subsection{DiPerna-Majda measures}
While Young measures \cite{y} successfully capture oscillatory behavior   of
sequences they completely miss concentrations. There are several tools how to deal with concentrations.
They can be considered as generalization of Young measures, see for example Alibert's and Bouchitt\'{e}'s approach  \cite{ab}, DiPerna's and Majda's treatment of concentrations \cite{DiPeMaj87a}, or  Fonseca's method described in \cite{fonseca}. An overview can be found in \cite{r,tartar1}.  Moreover, in many cases,
we are interested in oscillation/concentration effects generated by sequences of gradients. A characterization of Young measures generated by gradients was
completely given by Kinderlehrer and Pedregal \cite{k-p1,k-p}, cf. also \cite{pedregal}.
 The  first attempt to characterize both
oscillations and concentrations in sequences of gradients is due
to Fonseca, M\"{u}ller, and Pedregal \cite{fmp}. They dealt with a special situation of $\{g v(\nabla u_k)\}_{k\in\N}$ where $v$ is positively $p$-homogeneous, $u_k\in W^{1,p}(\O;\R^m)$, $p>1$, with  $g$ continuous and   vanishing on $\partial\O$.
Later on, a characterization of oscillation/concentration effects in terms of  DiPerna's and Majda's generalization
of Young measures was given in \cite{KaKru08a} for arbitrary integrands and in \cite{ifmk} for sequences living in the kernel of a first-order differential operator. Recently Kristensen and Rindler \cite{kristensen-rindler} characterized oscillation/concentration effects in the case $p=1$.

Let us take a complete (i.e. containing constants, separating points
from closed subsets and closed with respect to the Chebyshev norm)
separable  ring ${\cal R}$ of
continuous bounded functions $\R^{m\times n}\to\R$. It is known \cite[Sect.~3.12.21]{engelking} that there is a
one-to-one correspondence ${\cal R}\mapsto\b_{\cal R}\R^{m\times n}$ between such
rings and metrizable compactifications of $\R^{m\times n}$; by a compactification
we mean here a compact set, denoted by $\b_{\cal R}\R^{m\times n}$, into which
$\R^{m\times n}$ is embedded homeomorphically and densely. For simplicity, we
will not distinguish between $\R^{m\times n}$ and its image in $\b_{\cal R}\R^{m\times n}$. Similarly, we will not distinguish between elements of ${\cal R}$ and their unique continuous extensions on $\b_{\cal R}\R^{m\times n}$.

 Let $\s\in\rca(\bar\O)$ be a  positive Radon measure on a bounded domain $\O\subset\R^n$. A
mapping $\hat\nu:x\mapsto \hat\nu_x$ belongs to the
space $L^{\infty}_{\rm w}(\bar{\O},\s;\rca(\b_{\cal R} \R^{m\times n}))$ if it is weakly*  $\s$-measurable (i.e., for any $v_0\in C_0(\R^{m\times n})$, the mapping
$\bar\O\to\R:x\mapsto\int_{\b_{\cal R}\R^{m\times n}} v_0(s)\hat\nu_x(\d s)$ is $\s$-measurable in
the usual sense). If additionally
$\hat\nu_x\in\prca(\b_{\cal R}\R^{m\times n})$ for $\s$-a.a. $x\in\bar\O$
 the collection $\{\hat\nu_x\}_{x\in\bar{\O}}$ is the so-called
Young measure on $(\bar\O,\s)$ \cite{y}, see also
\cite{Ba89a,r,tartar1}.

DiPerna and Majda \cite{DiPeMaj87a} shown that having a bounded
sequence in $L^p(\O;\R^{m\times n})$ with $1\le p<+\infty$ and
$\O$ an open domain in $\Rn$, there exists its subsequence
(denoted by the same indices) a positive Radon measure
$\s\in\rca(\bar\O)$ and a Young measure  $\hat\nu:x\mapsto
\hat\nu_x$  on  $(\bar\O,\s)$ such that  $(\s,\hat\nu)$ is
attainable by a sequence $\{y_k\}_{k\in\N}\subset
L^p(\O;\R^{m\times n})$ in the sense that $\forall g\!\in\!
C(\bar\O)\ \forall v_0\!\in\!{\cal R}$:
\be\label{basic}\lim_{k\to\infty}\int_\O g(x)v(y_k(x))\d x =
\int_{\bar\O}\int_{\b_{\cal R}\R^{m\times
n}}g(x)v_0(s)\hat\nu_x(\d s)\s(\d x)\ , \ee where \[
v\in\ups:=\{v_0(1+|\cdot|^p);\ v_0\in{\cal R}\}.\]
 In particular,
putting $v_0=1\in{\cal R}$ in (\ref{basic}) we can see that
\be\label{measure} \lim_{k\to\infty}(1+|y_k|^p)\ =\ \s \ \ \ \
\mbox{ weakly* in }\ \rca(\bar\O)\ . \ee If (\ref{basic}) holds,
we say that $\{y_k\}_{\in\N}$ generates $(\sigma,\hat\nu)$. Let us
denote by ${\cal DM}^p_{\cal R}(\O;\R^{m\times n})$ the set of all
pairs $(\s,\hat\nu)\in\rca(\bar\O)\times L^{\infty}_{\rm
w}(\bar{\O},\s; \rca(\b_{\cal R} \R^{m\times n}))$ attainable by
sequences from $L^p(\O;\R^{m\times n})$; note that, taking $v_0=1$
in (\ref{basic}), one can see that these sequences must be
inevitably bounded in $L^p(\O;\R^{m\times n})$. We also denote by $\gcdm$ measures from $\cdm$ generated by a sequence of gradients of some bounded sequence in $W^{1,p}(\O;\R^m)$.  The explicit
description of the elements from ${\cal DM}^p_{\cal
R}(\O;\R^{m\times n})$, called DiPerna-Majda measures, for
unconstrained sequences  was given  in \cite[Theorem~2]{KruRou97a}.
In fact, it is easy to see that (\ref{basic}) can be also written in the form
\be\label{basic0}\lim_{k\to\infty}\int_\O h(x,y_k(x))\d x =
\int_{\bar\O}\int_{\b_{\cal R}\R^{m\times
n}}h_0(x,s)\hat\nu_x(\d s)\s(\d x)\ , \ee where $
h(x,s):=h_0(x,s)(1+|s|^p) $ and $h_0\in C(\bar\O\otimes\beta_{\cal R}\R^{m\times n})$.

We say that $\{y_k\}$ generates $(\sigma,\hat\nu)$ if (\ref{basic}) holds. Moreover, we denote $d_\sigma\in L^1(\O)$ the absolutely continuous (with respect to the Lebesgue measure) part of $\sigma$ in the Lebesgue decomposition of $\sigma$.

We will denote elements from $\cdm$ which are generated by $\{\nabla u_k\}_{k\in\N}$ for some bounded $\{u_k\}\subset W^{1,p}(\O;\R^m)$ by $\gcdm$.

\subsubsection{Compactification of $\R^{m\times n}$ by the sphere}
In what follows we will  work   mostly  with a particular compactification of $\R^{m\times n}$, namely, with the compactification by the sphere. We will consider
the following ring  of continuous bounded functions
\begin{equation}\label{spherecomp}
\begin{aligned}
{\cal S}:=\bigg\{ \, v_0\in C(\R^{m\times n})\,:\,\text{there exist}~ v_{0,0}\in C_0(\R^{m\times n})\, v_{0,1}\in C(S^{(m\times n)-1})\, \mbox{, and } c\in\R \mbox{ s.t. }&\\
 v_0(F) := c+ v_{0,0}(F)+v_{0,1}\left(\frac{F}{|F|}\right)
\frac{|F|^p}{1+|F|^p}\mbox { if $F\ne 0$ and }  v_0(0):=v_{0,0}(0) \,&\bigg\} ,
\end{aligned}
\end{equation}
where $S^{m\times n-1}$ denotes the $(mn-1)$-dimensional unit sphere in $\R^{m\times n}$. Then $\b_{\cal
S}\R^{m\times n}$ is homeomorphic to the unit ball $\overline{B(0,1)}\subset \R^{m\times n}$ via the mapping $d:\R^{m\times n}\to B(0,1)$, $d(s):=s/(1+|s|)$ for all $s\in\R^{m\times n}$. Note that $d(\R^{m\times n})$ is dense in $\overline{B(0,1)}$.

For any $v\in\upss$  there exists a continuous and positively $p$-homogeneous function $v_\infty:\R^{m\times n}\to\R$ (i.e. $v_\infty(\alpha F)=\alpha^p v_\infty(F)$ for all $\alpha\ge 0 $ and $ F\in\R^{m\times n}$) such that
\be\label{recessionf}
\lim_{|F|\to\infty}\frac{v(F)-v_\infty(F)}{|F|^p}=0\ .
\ee

Indeed, if $v_0$ is as in (\ref{spherecomp}) and $v=v_0(1+|\cdot|^p)$  then set
$$v_\infty(F):=\left(c+v_{0,1}\left(\frac{F}{|F|}\right)\right)|F|^p\mbox{ for $F\in\R^{m\times n}\setminus\{0\}$.} $$
By continuity we define $v_\infty(0):=0$. It is easy to see that $v_\infty$ satisfies (\ref{recessionf}).
Such $v_\infty$ is called the {\it recession function} of $v$.

The following two results were proven in \cite{Kru10a}.

\begin{lemma}\label{equiintegrability}
Let $1\le p<+\infty$,  $0\le h_0\in C(\bar\O\times\beta_{\cal S}\R^{m\times n})$, let $h(x,F):=h_0(x,F)(1+|F|^p)$, and let $\{u_k\}\subset W^{1,p}(\O;\R^m)$ be a bounded sequence with $\{\nabla u_k\}_{k\in\N}\subset L^p(\O;\R^{m\times n})$ generating
$(\sigma,\hat\nu)\in {\cal DM}^p_{\cal S}(\O;\R^m)$. Then
$$
	\text{$\{h(x,\nabla u_k)\}_{k\in\N}$ is weakly relatively compact in $L^1(\O)$}
$$
if and only if
\be\label{wrc}
\int_{\bar\O}\int_{\beta_{\cal S}\R^{m\times n}\setminus\R^{m\times n}} h_0(x,F)\hat\nu_x(\md F)\sigma(\md x)=0\ .
\ee

\end{lemma}
\begin{remark}\label{rem:equiintegrability}
In Lemma~\ref{equiintegrability}, we assumed that $h_0$ and, consequently, $h$ are non-negative, but this assumption can be relaxed. For the assertion of the lemma to hold true, it actually suffices to have that $h(x,\nabla u_k)\geq 0$ for every $k$ and a.e.~$x\in\O$. This can easily be seen by applying Lemma~\ref{equiintegrability} with $h^+$ (the positive part) instead of $h$.
\end{remark}

\bigskip

\begin{theorem}\label{suff-nec}
Let  $\O\subset\R^n$ be a bounded domain with boundary of class $C^1$,
$1<p<+\infty$, and $(\sigma,\hat\nu)\in\cdms$. Then there is
 a bounded sequence
$\{u_k\}_{k\in\N}\subset W^{1,p}(\O;\R^m)$ such that $\{\nabla
u_k\}_{k\in\N}$ generates $(\sigma,\hat\nu)$ if and only if the
following four conditions hold:
\be\label{firstmoment7} \exists u\in W^{1,p}(\O;\R^m)\ \mbox{ such that  for a.a. $x\in\O$:  }  \nabla
u(x)=d_\sigma(x)\int_{\beta_{\cal S}\R^{m\times n}}\frac{F}{1+|F|^p}\hat\nu_x(\d
F)\ ,
\ee
for almost all $x\in\O$ and for all  $v\in\upss$  the
following inequality  is fulfilled
\be\label{qc7} Qv(\nabla
u(x))\le d_\sigma(x)\int_{\beta_{\cal S}\R^{m\times
n}}\frac{v(F)}{1+|F|^p}\hat\nu_x(\d F)\ ,
\ee
for $\sigma$-almost
all $x\in\O$ and all $v\in\upss$ with $Qv_\infty>-\infty$  it holds
that
\be\label{rem7}
 0\le  \int_{\beta_{{\cal S}}\R^{m\times n}\setminus\R^{m\times n}}\frac{v(F)}{1+|F|^p}\hat\nu_x(\md F)\ ,
\ee
and for $\sigma$-almost all $x\in\partial\O$ with the outer unit normal to the boundary  $\varrho(x)$  and all  $v\in\upss$ with $Q_{b,\varrho(x)}v_\infty(0)>-\infty$ it holds that
\be\label{bd7}
0\le  \int_{\beta_{{\cal S}}\R^{m\times n}\setminus\R^{m\times n}}\frac{v(F)}{1+|F|^p}\hat\nu_x(\md F)\ .
\ee
\end{theorem}
\begin{remark} If the traces of $\{u_k\}$ are fixed near some $x\in \partial \O$ and coincide with the trace of $u$, i.e., $u_k=u$ 
in the sense of trace on $\partial\Omega$, see e.g. \cite{kujofu}, then
condition (\ref{bd7}) holds for a bigger class of admissible $v$, namely, all $v\in\upss$ with $Qv>-\infty$. This can be inferred from \cite[above Remark 3.9]{KaKru08a}.
\end{remark}
The theorem can be extended (with arguments analogous to case of Young measures as presented in \cite{Ka97a}) to allow $x$-dependent test functions (instead of $v$) in \eqref{qc7}--\eqref{bd7}:
\begin{corollary}\label{cor:suff-nec}
In the situation of Theorem~\ref{suff-nec}, if $(\sigma,\hat\nu)$ is generated by $\{\nabla
u_k\}_{k\in\N}$, with a bounded sequence
$\{u_k\}_{k\in\N}\subset W^{1,p}(\O;\R^m)$, then in addition to \eqref{firstmoment7},
the following three conditions hold for all functions $h$ of the form $h(x,F):=h_0(x,F) (1+\abs{F}^p$ with some $h_0\in C(\overline{\O}\times \coms)$:

For almost all $x\in\O$ and all $h$, we have that
\be\label{qc7B} Qh(x,\nabla
u(x))\le d_\sigma(x)\int_{\beta_{\cal S}\R^{m\times
n}}\frac{h(x,F)}{1+|F|^p}\hat\nu_x(\d F)\ ,
\ee
for $\sigma$-almost
all $x\in\O$ and all $h$ with $Q h(x,\cdot)>-\infty$, it holds
that
\be\label{rem7B}
 0\le  \int_{\beta_{{\cal S}}\R^{m\times n}\setminus\R^{m\times n}}\frac{h(x,F)}{1+|F|^p}\hat\nu_x(\md F)\ ,
\ee
and for $\sigma$-almost all $x\in\partial\O$ with the outer unit normal to the boundary  $\varrho(x)$  and all $h$ with $[Q_{b,\varrho(x)} h_\infty(x,\cdot)](0)>-\infty$, where $h_\infty$ is the recession function of $h$ with respect to the second variable, we have that
\be\label{bd7B}
0\le  \int_{\beta_{{\cal S}}\R^{m\times n}\setminus\R^{m\times n}}\frac{h(x,F)}{1+|F|^p}\hat\nu_x(\md F)\ .
\ee
\end{corollary}

\section{Weak continuity up to the boundary}

\begin{theorem}\label{th:weakcontinuityup}
Let $m,n\in\N$ with $n\geq 2$, let $\Omega\subset \R^n$ be open and bounded with boundary of class $C^1$, and let
$f:\overline{\Omega}\times \R^{m\times n}\to \R$ be a continuous function. In addition, suppose that for every $x\in \Omega$, $f(x,\cdot)$ is a null Lagrangian and
for every $x\in \partial\Omega$, $f(x,\cdot)$ is a null Lagrangian at the boundary with respect to $\varrho(x)$, the outer normal to $\partial\Omega$ at $x$. Hence, by Theorem \ref{thm:bnulllag}, $f(x,\cdot)$ is a polynomial, whose degree we denote by $d_f(x)$.
Finally, let $p\in (1,\infty)$ with $p\geq d_f(x)$ for every $x\in\overline{\Omega}$
and let $(u_k)\subset W^{1,p}(\Omega;\R^m)$ be a sequence with $u_k\rightharpoonup u$ weakly in $W^{1,p}$.
If
$$
	f(x,\nabla u_k(x))\geq 0~~\text{for every $k\in\N$ and a.e.~$x\in\Omega$,}
$$
then $f(\cdot,\nabla u_n)\rightharpoonup f(\cdot,\nabla u)$ weakly in $L^1(\O)$.
\end{theorem}
The proof relies on the following auxiliary result, justifying that $h=f$ is admissible a test function in Corollary~\ref{cor:suff-nec}:
\begin{lemma}\label{lem:parpoly}
Let $p\geq 0$ and suppose that $f:\overline{\O}\times \R^{m\times n}\to \R$ is continuous, and for each $x\in \overline{\O}$,
$F\mapsto f(x,F)$ is a polynomial of degree at most $p$. Then
$$
	f_0:\overline{\O}\times \R^{m\times n}\to \R,\quad f_0(x,F):=\frac{f(x,F)}{1+\abs{F}^p}
$$
has a continuous extension to $\overline{\O}\times \coms$.
\end{lemma}
\begin{proof}
Let $d:=\max_{x\in \overline{\O}}d_f(x)$ denote the maximal degree of $f$.
For every $x$, we split
$$	
	f(x,F)=a_0(x,F)+a_1(x,F)+\ldots+a_d(x,F),
$$
where for each $i$, $a_i(x,\cdot)$ is a positively $i$-homogeneous polynomial.
We first claim that $a_i$ is continuous on $\overline{\O}\times \R^{m\times n}$ for each $i$, which we prove
by induction with respect to $d$.
If $d=0$, the continuity of $a_0=f$ is trivial.
If $d>1$, since $f$ is continuous, so is
$$
	g(x,F):=2^d f(x,F)-f(x,2F)=\sum_{i=0}^{d-1} (2^d-2^i)a_i(x,F),
$$
whose maximal degree is (at most) $d-1$. By assumption of the induction,
we obtain that $(2^d-2^i)a_i$ and thus $a_i$ is continuous for each $i=1,\ldots,d-1$. As a consequence, $a_d=f-\sum_{i=1}^{d-1}a_i$ is continuous as well.

Due to the preceding observation, we may now assume that $f=a_i$ for some $i$, i.e., $f$ is positively $i$-homogeneous in its second variable. It is enough to obtain an continuous extension of $f_0$ for $F$ outside a fixed ball. For any $F$ with $\abs{F}> 0$, we have that
$$
	f_0(x,F)=\frac{f(x,F)}{1+\abs{F}^p}=\frac{\abs{F}^i}{1+\abs{F}^p} f\Big(x,\frac{F}{\abs{F}}\Big).
$$
This clearly has a continuous extension to $\overline{\O}\times \coms$.
%
%the first factor has a fixed limit as $\abs{F}\to \infty$, and
%$f$ is uniformly continuous on $\overline{\O}\times S^{nm-1}$ (the latter denoting the unit sphere in $\R^{m\times n}$).
\end{proof}
\bigskip

\noindent
\begin{proof}[Proof of Theorem~\ref{th:weakcontinuityup}]
%Recall that by Theorem \ref{thm:bnulllag}, $f(x,\cdot)$ really is a polynomial (of degree $d_f(x)\leq \min(m,n)$ for $x\in\Omega$ %and $d_f(x)\leq \min(m,n-1)$ for $x\in\partial\Omega$).
Let $(\sigma,\hat{\nu})$ be the DiPerna-Majda measure generated by (a subsequence of) $(\nabla u_k)$. In particular,
\begin{align}\label{twcu-1}
	\int_\Omega \varphi(x) f(x,\nabla u_k)\,\d x\underset{k\to\infty}{\longrightarrow}
	\int_{\bar\O}\int_{\coms} \varphi(x)\frac{f(x,F)}{1+\abs{F}^p}\,\hat{\nu}_x(\d F)\,\sigma(\d x)
\end{align}
for every $\varphi \in C(\bar\O)$. By Lemma~\ref{lem:parpoly}, $h:=\pm f$ is admissible in
the conditions \eqref{qc7B},\eqref{rem7B} and \eqref{bd7B} in Corollary~\ref{cor:suff-nec}, which also means that all three inequalities actually are equalities.
By \eqref{rem7B} and \eqref{bd7B}, we obtain that
\begin{align}\label{twcu-2}
	\int_{\bar\O}\int_{\beta_{\cal S}\R^{m\times n}\setminus\R^{m\times n}} \frac{f(x,F)}{1+\abs{F}^p} \hat\nu_x(\md F)\sigma(\md x)=0.
\end{align}
Using this together with \eqref{qc7B}, the right hand side in \eqref{twcu-1} can be expressed as
$$
\begin{aligned}
		\int_{\bar\O}\int_{\coms} \varphi(x)\frac{f(x,F)}{1+\abs{F}^p}\,\hat{\nu}_x(\d F)\,\sigma(\d x)
		&=\int_{\bar\O}\int_{\R^{m\times n}} \varphi(x)\frac{f(x,F)}{1+\abs{F}^p}\,\hat{\nu}_x(\d F)\,\sigma(\d x) \\
		&=\int_{\bar\O}\int_{\R^{m\times n}} \varphi(x)\frac{f(x,F)}{1+\abs{F}^p}\,\hat{\nu}_x(\d F)\,\d_\sigma(x)\, \d x\\
		&=\int_\Omega \varphi(x) f(x,\nabla u)\,\d x
\end{aligned}
$$
Consequently, \eqref{twcu-1} implies that
$f(\cdot,\nabla u_k)\to f(\cdot,\nabla u)$ weakly* in $(C(\bar\O))'=\rca(\bar\O)$. Finally, if $f(x,\nabla u_k(x))\ge 0$ for almost all $x\in\O$ and all $k\in\N$, then
 $f(\cdot,\nabla u_k)\wto f(\cdot,\nabla u)$ in $L^1(\O)$ by Lemma~\ref{equiintegrability} and Remark~\ref{rem:equiintegrability}, using \eqref{twcu-2}.
\end{proof}

The following result evokes M\"{u}ller's  generalization  \cite{Mue90a} of Ball's result  \cite{Ball:77a}. In our setting, however, we can drop nonnegativity of the integrand.
The condition  $f(\cdot,\nabla u)\ge 0$ can be seen as a kind of ``orientation-preservation''. We refer to \cite{CiGoMa:11a} for elasticity of shells including a normal-orientation condition.

\bigskip

\begin{theorem}
Let $h:\O\times\R\to\R\cup\{+\infty\}$ be such that $h(\cdot, s)$ is measurable for all $s\in\R$ and  $h(x,\cdot)$ is  convex for almost all $x\in\O$.  Let $f$ and $d_f$  be as in Theorem~\ref{th:weakcontinuityup}. Then $I(u):=\int_\O h(x,f(x,\nabla u(x))\,\md x$ is weakly lower semicontinuous on the set $\{u\in W^{1,p}(\O;\R^m); f(\cdot,\nabla u)\ge 0 \mbox{ in }\O\}$.
\end{theorem}
\bigskip

{\it Proof.}
The proof is standard.
\hfill $\Box$

\section{Higher Integrability}

By a result of  S.~M\"uller \cite{Mue89a, Mue90a}, for any bounded sequence $(u_k\subset W^{1,n}(\Omega;\R^N)$ with $\det \nabla u_k\geq 0$ a.e.~in $\Omega$, $\det \nabla u_k$ is locally bounded in the class $L \log L$, i.e.,
$$
	\sup_k \int_K \gamma(\det \nabla u_k)\,\d x<\infty , \quad\text{with $\gamma(s):=s\ln^+ s$},
$$
for every $K\subset \subset \O$, with $\ln^+$ denoting the positive part of the logarithm. It is natural to ask whether an analogous result holds for a null Lagrangian at the boundary in place of the determinant, up to the boundary (i.e., with $\Omega$ instead of $K$). The example closest to  M\"uller's original result for the determinant is the function $\det'$, given by
$$
	\det{}' \xi:=\det (\xi_{ij})_{i,j=1,\ldots,n-1},\quad \text{for $\xi\in \R^{(n-1)\times n}$}.
$$
This is a null Lagrangian at the boundary, at every boundary point with the normal $\varrho=e_n$. A strict analogue for the estimates in \cite{Mue89a, Mue90a} in this case would be an inequality as follows:
\begin{eqnarray}\label{mulext}
	\sup_k \int_K \gamma(\det{}' \nabla u_k)\,\d x\leq C\Big(K, \| u\|_{W^{1,n-1}(\Omega;\R^{n-1})}\Big)
 \quad\text{with $\gamma(s):=s\ln^+ s$},
\end{eqnarray}
with a continuous function $C(K,\cdot)$, for every compact $K\subset \bar\O$ having a positive distance to the set $\{x\in\partial\Omega;\,\rho(x)\neq e_n\}$ where $\rho(x)$ denotes the outer unit normal to $\partial\O$ at $x$ . However, it seems that it is not possible to extend  M\"uller's proof to this case, at least not in a straightforward way, and the validity of \eqref{mulext} remains an open problem.

On the other hand, $\det{}' \nabla u_k$ only depends on the derivatives of $u_k$ with respect to first $n-1$ variables. The anisotropic space $L^{n-1}((0,\frac{1}{2});W^{1,n-1}((0,1)^{n-1};\R^{n-1}))$ suffices to ensure integrability of $\det{}' \nabla u_k$, which makes it a natural alternative to the isotropic space $W^{1,n-1}(\Omega;\R^{n-1})$ used above.
It turns out that the analogue of \eqref{mulext} with the anisotropic norm on the right hand side
fails to hold even in the interior, as illustrated by the example below, an extension of Counterexample 7.2 in \cite{Mue90a}.
More precisely, we show that  one cannot expect an inequality of the following form:
\begin{eqnarray}\label{mulext2}
	\sup_k \int_K \gamma(\det{}' \nabla u_k)\,\d x\leq C\Big(K, \| u\|_{L^{n-1}( (-1,1); W^{1,n-1}((0,1)^{n-1};\R^{n-1})}\Big)
 \quad\text{with $\gamma(s):=s\ln^+ s$},
\end{eqnarray}
with a continuous function $C(K,\cdot)$.
\begin{example}
For $n\geq 2$ consider
$$
	\det{}':\R^{(n-1)\times n}\to \R, \quad \det{}' \xi:=\det (\xi_{ij})_{i,j=1,\ldots,n-1},
$$
which is a boundary Null Lagrangian with respect to the normal $\varrho=\pm e_n$ (the $n$-th unit vector),
together with the sequence $(u_k)$ defined as
$$
	u_k(x',x_n):=g(x_n)h_k(\abs{x'})\frac{x'}{\abs{x'}},~~(x',x_n)\in Q:=(0,1)^{n-1}\times (-1,1),
$$
where
$$
	\begin{aligned}
		&g(t):=\left[|t| \ln^2 (|t|)\right]^{-\frac{1}{n-1}},\\
		%\left\{\begin{array}{ll}
		%\end{array}\right.,\\
		&h_k(r):=\left\{\begin{array}{ll}
		k (\ln k)^{-\frac{1}{n-1}} r&~~\text{if $r<\frac{1}{k}$,}\\
		(\ln k)^{-\frac{1}{n-1}}&~~\text{else.}
		\end{array}\right.\\
	\end{aligned}
$$
In this case, one can check (cf.~\cite{Mue90a}) that  $\det{}' (\nabla u_k)=g(x_n)^{n-1}k^{n-1}(\ln k)^{-1}$ for $\abs{x'}<\frac{1}{k}$ and $\det{}' (\nabla u_k)=0$ elsewhere. In particular $\det{}' (\nabla u_k)\geq 0$ a.e.~in $Q$, for every $k$. In addition,
$(u_k)\subset L^{n-1}((0,1);W^{1,n-1}((0,1)^{n-1};\R^{n-1}))$ is bounded, i.e.,
$$
	\sup_k \int_Q \abs{\nabla' u_k}^{n-1}\,\d x
	%=\sup_k \frac{1}{\ln 2}\int_{(0,1)^{n-1}} \abs{\nabla' u_k}^{n-1}\,\d x'
	<\infty,
$$
where $\nabla'$ denotes the gradient with respect to the first $n-1$ variables. But
for every $k$, the leading term in $\int_{(0,1)^{n-1}} \gamma(\det{}' \nabla u_k(x',x_n))\,\d x'$ for $x_n$ near zero is of the form
$$
	\frac{-1}{\abs{x_n} \ln \left(\abs{x_n}\right)\ln k},
$$
which is not integrable near $x_n=0$, and consequently,
$$
	\int_{K_\eps} \gamma(\det{}' \nabla u_k)\,\d x=+\infty\quad
	\text{with $\gamma(s):=s\ln^+ s$ and $K_\eps:=[\eps,1-\eps]^{n-1}\times [0,\eps]$}.
$$
Hence inequality \eqref{mulext2} cannot hold.
\end{example}
%\begin{remark}
%Clearly, $\det{}' \nabla u_k$ only depends on the derivatives of $u_k$ with respect to first $n-1$ variables. The anisotropic %space $L^{n-1}((0,\frac{1}{2});W^{1,n-1}((0,1)^{n-1};\R^{n-1}))$ suffices to ensure integrability of $\det{}' \nabla u_k$, which %makes it a natural choice. We do not know if it is possible to find another sequence with the same properties as outlined in the %example, which is even bounded in
%$W^{1,n-1}(Q;\R^{n-1})$. As a matter of fact, this cannot be achieved by simply truncating the singularity of  $g(x_n)$ near %$x_n=0$ at a $k$-dependent level.
%\end{remark}

\bigskip

{\bf Acknowledgment:} The work was conducted during repeated mutual visits of the authors at the Universities of Cologne and Warsaw and at the Institute of Information Theory and Automation in Prague. The hospitality and support of all these institutions is gratefully acknowledged.

%{\bf Acknowledgment:} The work was conducted during repeated stays of MK at the Universities of Cologne and Warsaw and several %stays  of AK and SK in the Institute of Information Theory and Automation in Prague.  The hospitality and support of all these  %institutions is gratefully acknowledged. \red{ I- A.K. was also in Cologne}

%\bibliographystyle{plain}
%\bibliography{lagrangian}

\end{sloppypar}

\end{document}